\documentclass{llncs}

\usepackage{hyperref}

\usepackage{cite}
\usepackage{graphicx}
\usepackage{framed}
\usepackage{amsfonts}
\usepackage{amsmath}
\usepackage{color}
\usepackage{algorithm}
\usepackage[noend]{algpseudocode}
\definecolor{mygreen}{rgb}{0,0.2,0}
\definecolor{mygray}{rgb}{0.95,0.95,0.95}
\definecolor{mymauve}{rgb}{0.58,0,0.82}

\definecolor{lbcolor}{rgb}{0.95,0.95,0.95}

\usepackage{listings,newtxtt}
\begin{document}

\title{Teaching algebraic curves for gifted learners at age 11 by using LEGO linkages and GeoGebra}
\author{Zolt\'an Kov\'acs\orcidID{0000-0003-2512-5793}}
\institute{
The Private University College of Education of the Diocese of Linz\\
Salesianumweg 3, A-4020 Linz, Austria\\
\email{zoltan@geogebra.org}
}

\maketitle              

\begin{abstract}
A summary of an experimental course on algebraic curves is given
that was held for young learners at age 11. The course was a part
of Epsilon camp, a program designed for very gifted students who have
already demonstrated high interest in studying mathematics. Prerequisites for
the course were mastery of Algebra I and at least one preliminary year in a prior
Epsilon camp. The summary gives an overview of the flow of teaching,
the achieved results and some evaluation of the given feedback.


\end{abstract}
\section{Introduction}

Epsilon Camp (\url{https://epsiloncamp.org}) is an annual meeting of young learners in the United States that has been held
at different campus locations since 2011 for two weeks during summer vacation.
It was founded by George Reuben Thomas
who is currently the executive director of the camp. The camp program is organized
by an academic director who follows the mission statement of the camp. In a nutshell,
the young learners are divided into groups that reflect their age and preliminary skills.

Campers are placed into groups of 10--12 campers with
a dedicated counselor, who is usually a university student in mathematics. Camper groups within each level do not imply
any difference in ability. Camper levels are currently as follows:

\begin{enumerate}
        \item Pythagoras Level: 7- and 8-year-olds.
        \item Euclid Level: 9- to 11-year-olds who are new to camp or 9- to 11-year-olds who have previously completed any part of Pythagoras Level.
        \item Gauss Level: 10- to 12-year-olds who have completed the Euclid Level.
        \item Conway Level: 11- to 12-year-olds who have completed the Gauss Level.
\end{enumerate}

Typical curriculum topics for campers have included introductory topics to advanced mathematics,
including number theory, methods of proofs, voting theory, set theory, Euclidean
geometry, projective geometry, or hyperbolic geometry, among others. These topics
can vary from year to year, but the basic concept is that for higher levels some
non-introductory topics are also included.

In this paper we sketch up the flow of a course given at Epsilon Camp 2019,
introduced for the Gauss Level, held 14--28 July 2019 at University of Colorado, Colorado Springs.

\section{Course description}
Eleven students attended the course: 4 girls and 7 boys. The youngest learner was 10 and
the oldest 12 years old.

The course consisted of 10 classes, 80 minutes each. The students were taught
in a university classroom with whiteboard and projector access. The students
worked only with paper and pencil, and a LEGO construction set that is
described at \cite{gh-lego-linkages}. Using calculators for the students in the camp was discouraged,
however in this course at some points it was still allowed.

The basic idea of the course was to explain some basic concepts of algebraic curves,
including classification of degree 1 and 2 curves, and identifying higher degree
curves that can be drawn by planar linkages. The main research question was
to learn whether Watt's linkage could produce a straight linear motion or not, and if not,
could the situation be improved. To achieve this goal, the learners did experiments
in multivariate polynomial division with respect to various term orderings,
and performed factorizations over the rationals. The basic concept of
the Nullstellensatz \cite{nullstellensatz} was presented and proven for linear factors.
The idea of eliminating variables from a system of algebraic equations was
shown with a concrete example by using a web user interface of the Singular \cite{Singular}
computer algebra system to illustrate Buchberger's algorithm \cite{gb-en}.

The mathematical relationship between the structure of a LEGO linkage and an algebraic
equation system was then explained. By
using elimination in GeoGebra \cite{gg5}, the motion produced by the linkage was computed
with computer use. The need to avoid degeneracy and a possible workaround was
discussed by understanding Rabinowitsch's trick \cite{Rabinowitsch1929}on defining non-equalities
to describe Agnesi's witch. Finally, after introducing inversion, Peaucellier's linkage
(see, e.g.~\cite{Kempe1877}) was studied, and the conclusion was made that it indeed produces a straight line motion.

The course was highly inspired by the book \cite{bryantsangwin} and used several
concepts from its chapter ``How to draw a straight line?''. Preliminary ideas
from the papers \cite{lego,ucmlego,jkulego,vmt2,JSC-linkages,aplimat-inv,lego-steam} were extensively used
during the course.

\section{The used didactic method}
Epsilon Camp's mission statement discouraged teaching
in a definition--theorem--proof style. Instead, a real life problem (namely, ``Can Watt's
linkage produce a straight line motion?'') was introduced, and the underlying theory
was built with extensive help of the students. The students' natural curiosity to learn
what kind of curve is generated from a given formula and vice versa, led to
natural discovery of new knowledge. The learning process was facilitated through
the Socratic method by the teacher: many questions were raised in order to
follow the expected line of reasoning.

Homework helped to deepen knowledge or to prepare next day's activity. Since each
group of children was learning three different subjects and the available time
for doing homework was 70 minutes per day, it was not expected that students
would spend more than 20 minutes to solve all homework exercises. On the other hand,
some students used some extra time to finalize their assignments and turned them
in---sometimes with standards of high quality. 

Some assignments were purely mathematical; others were about to build a LEGO construction
and draw the geometric curve it produces. One homework was about plotting a graph
of a two-variate polynomial.

\section{The course in detail}

In this section a detailed description of the course follows. A daily step-by-step
explanation of the topic is presented. 

\subsection*{Day 1}

The main problem setting about Watt's linkage was given, by building a compass
first and then Watt's linkage. (See \cite{gh-lego-linkages}, linkages
\texttt{compass} and \texttt{Watt}.)

A general survey of the students' knowledge was performed. It turned out that
the students were familiar with linear equations and various forms of them,
and they knew how to plot linear functions and read off an algebraic form
of a line if two points of it is given.

A challenge was given to students to try to identify quadratic curves.
The equation of a parabola in form $y=ax^2+bx+c$ was well-known.
The equation of a circle was also known, and some possible algebraic forms
of a hyperbola were discussed. As homework, further types of quadratic curves
were to be identified (see Fig.~\ref{hw-day1} for a turned-in homework).

\begin{figure}
\begin{center}
\includegraphics[width=0.5\textwidth]{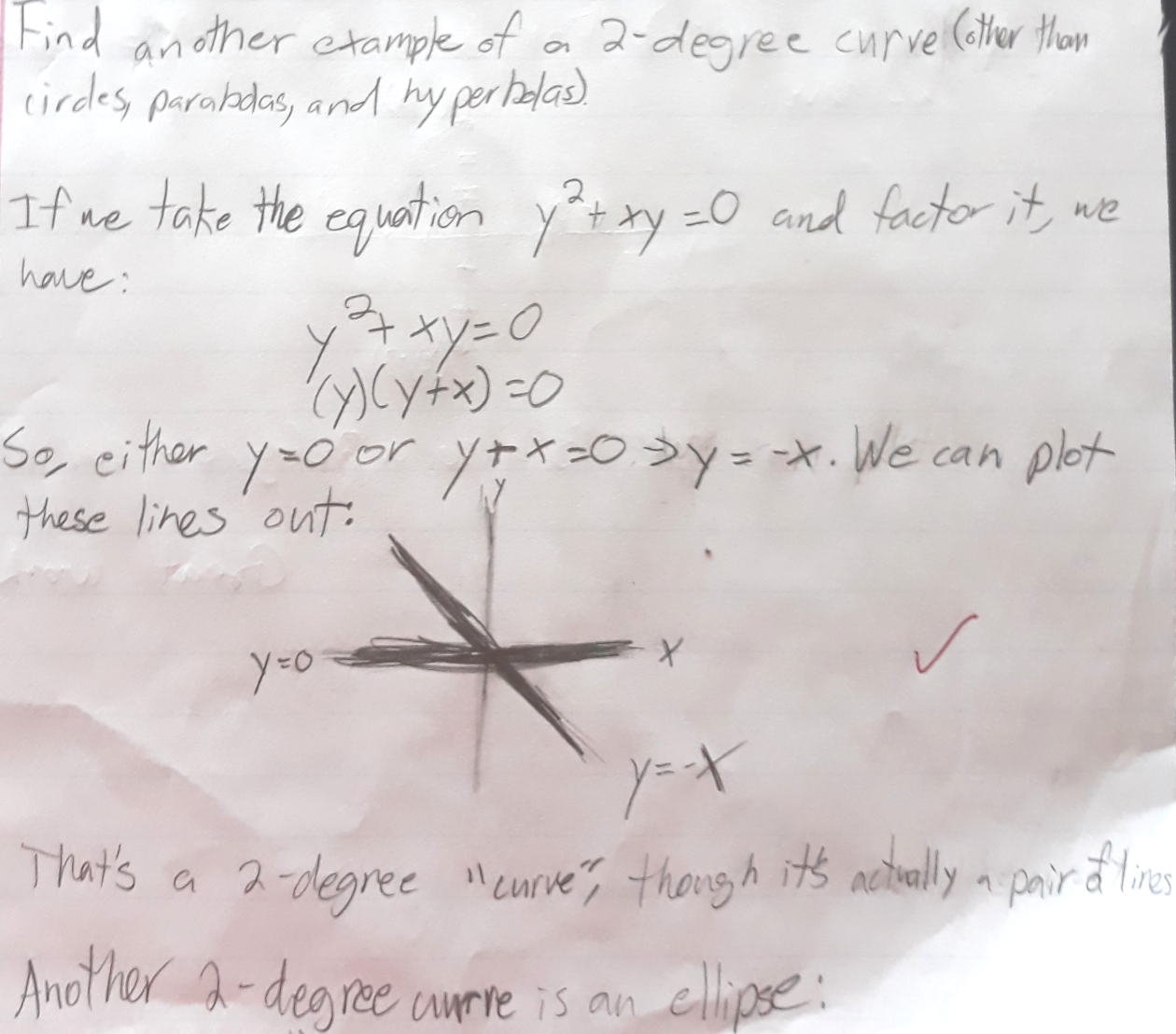}
\end{center}
\caption{A part of a student's worked-out homework for Day 1}
\label{hw-day1}
\end{figure}

\subsection*{Day 2}

A classification theorem of quadratic curves was stated (without proof):

\begin{theorem}\label{quad-class}
An equation of form $ax^2+by^2+cxy+dx+ey+f=0$ corresponds to either
a parabola, a circle, a hyperbola\footnote{To visualize a hyperbola
the GeoGebra applet \url{https://www.geogebra.org/m/AGT66pre} was used.}, an ellipse, two lines, a point
or the empty set.
\end{theorem}

After some discussion, the prerequisite $a\neq0$ or $b\neq0$ or $c\neq0$ was added.

Without giving a proof, a GeoGebra applet \cite{conicsections-irina} was
used to explain that the set of conic sections is the same as the set of quadratics.

We learned how difficult the factorization of a two-variate quadratic polynomial
can be, after computing the expansion of a product of two linear two-variate polynomials.
As homework a factorization exercise was given. It was rather challenging---two students
found the factors by solving an equation system of the coefficients in a combinatorial
way (see Fig.~\ref{hw-day2-jon} and \ref{hw-day2-noah} on a short and an explained
solution by the students), while another possible solution was discussed by finding some points on the
curve explicitly and then attempting to find lines that lie on those points (again,
combinatorially, see Fig.~\ref{hw-day2}).

\begin{figure}
\begin{center}
\includegraphics[width=0.5\textwidth]{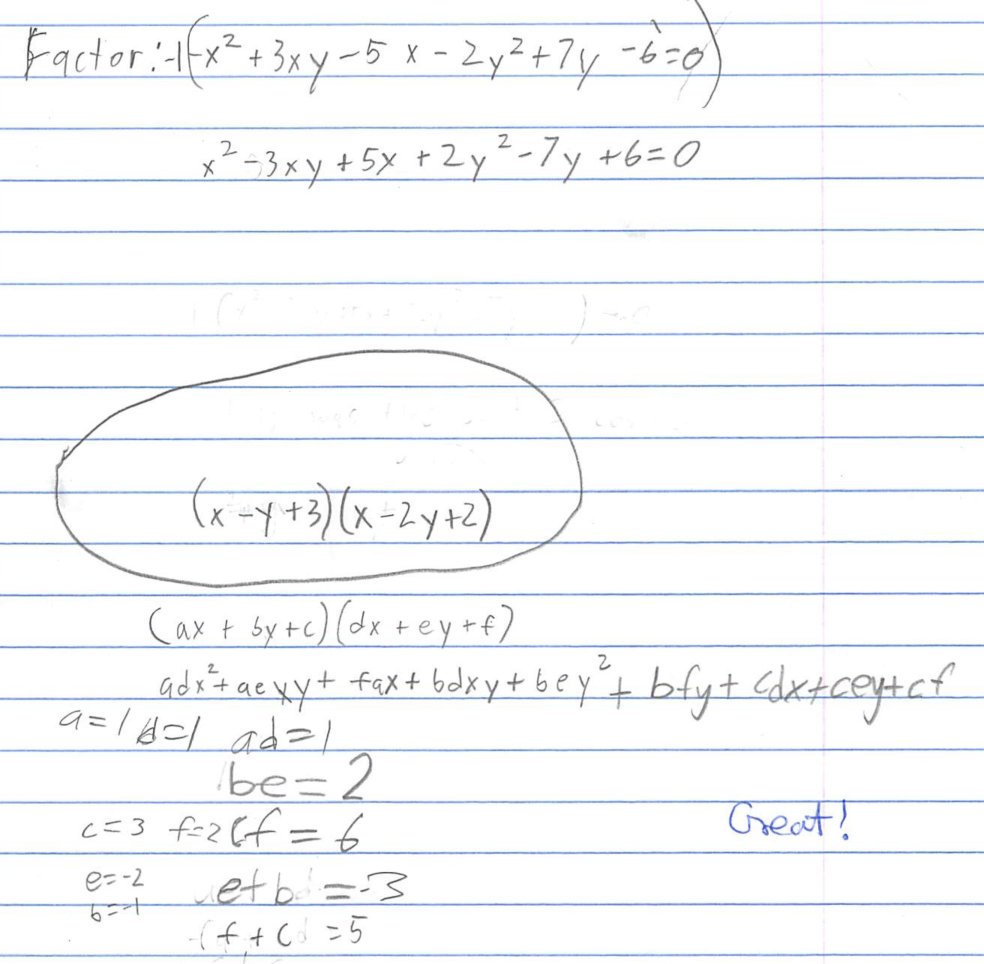}
\end{center}
\caption{A student's homework for Day 2}
\label{hw-day2-jon}
\end{figure}

\begin{figure}
\begin{center}
\includegraphics[width=0.7\textwidth]{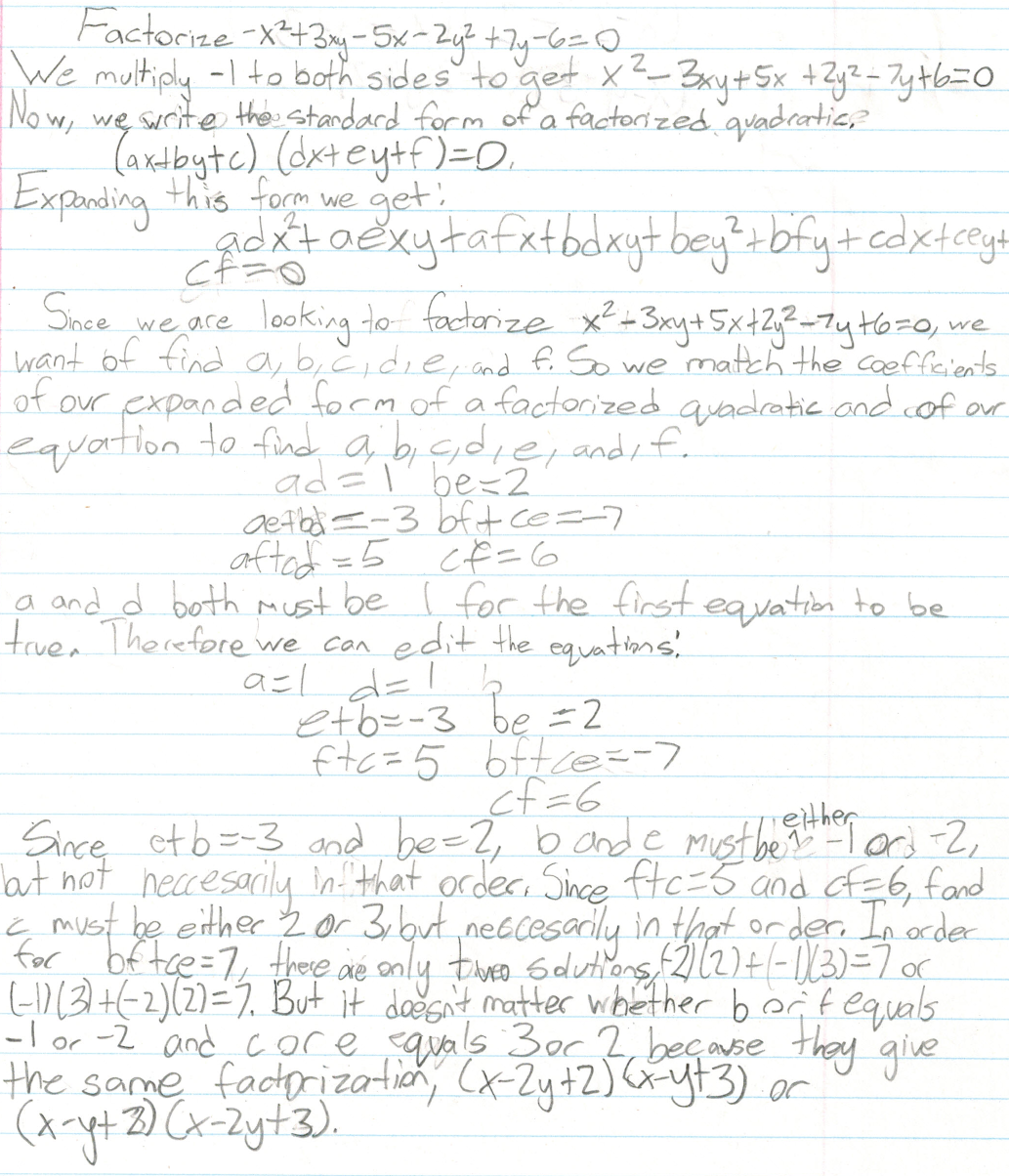}
\end{center}
\caption{A part of another student's homework for Day 2---the full explanation was provided later
on the teacher's request}
\label{hw-day2-noah}
\end{figure}

\begin{figure}
\begin{center}
\includegraphics[width=0.5\textwidth]{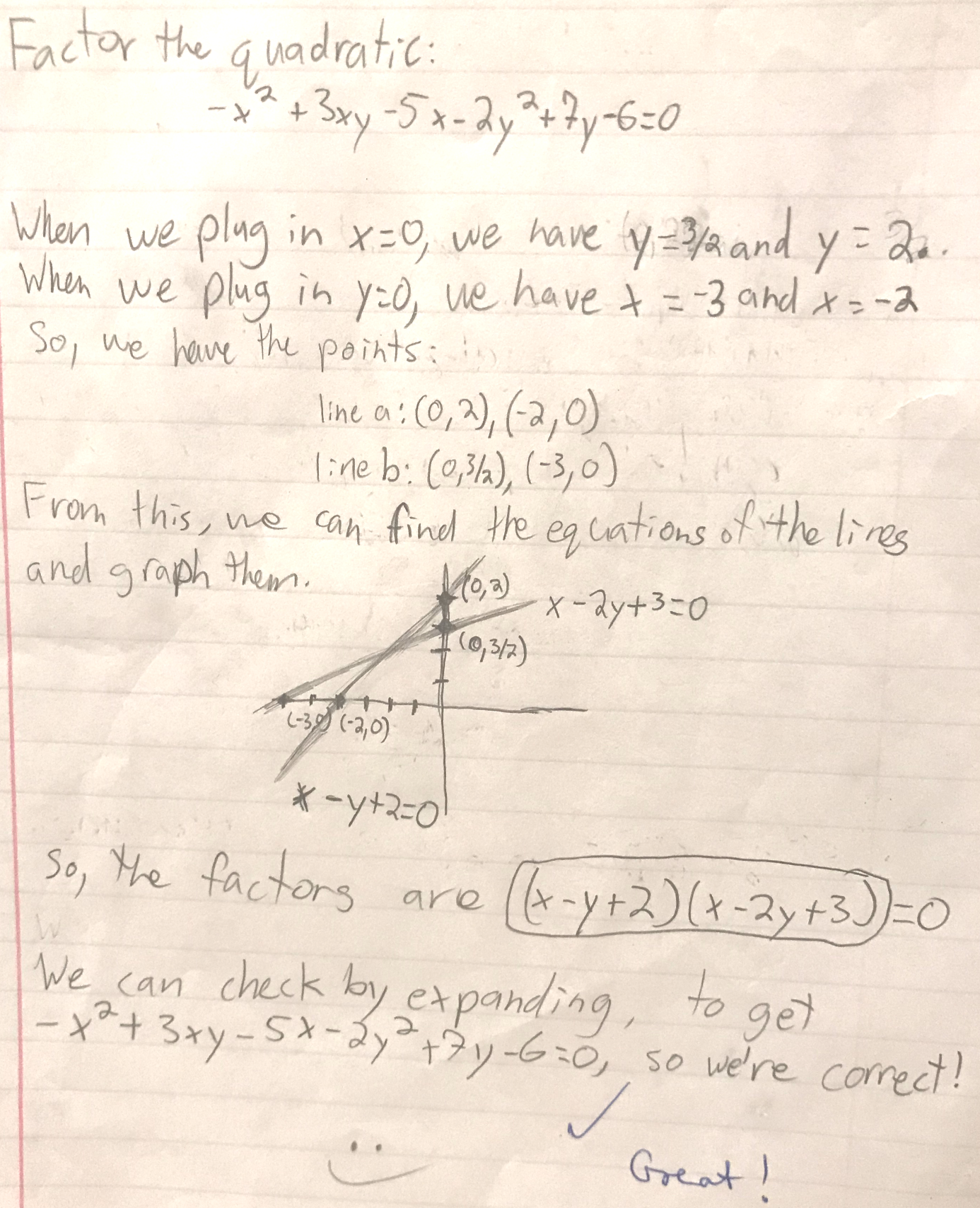}
\end{center}
\caption{Another student's homework for Day 2}
\label{hw-day2}
\end{figure}

The Zero Product Theorem was implicitly mentioned:
\begin{theorem}\label{th-factors}
Given two two-variate polynomials $p$ and $q$. Then the curve $p\cdot q=0$ is a union
of the curves $p=0$ and $q=0$.
\end{theorem}

\subsection*{Day 3}

The graph of some cubic polynomials was shown by using \url{https://en.wikipedia.org/wiki/Cubic_plane_curve}.

Some integer points of the curve $x^3+y^2-17=0$ were found. As homework, a graph of
the curve was to be plotted. An open question was whether cubic curves can be split to
factors of lower degree curves. For this purpose, multivariate long division was
introduced---the students already had some knowledge on dividing polynomials with remainders.

As a final fun activity students named various polynomials and they were plotted by GeoGebra
on the projector. A difference was made which curves are not algebraic.

\subsection*{Day 4}

Four different term orderings were introduced (lexicographical, reversed lexicographical,
degree lexicographical and degree reverse lexicographical) for two-variate polynomials.
A definition of admissible ordering was given, but not directly used in the rest of the course.
A challenge question was however given to show that the lexicographical term ordering
is admissible.

For example, lexicographical term ordering was explained by using Fig.~\ref{termordering}.

\begin{figure}
\begin{center}
\includegraphics[width=0.3\textwidth]{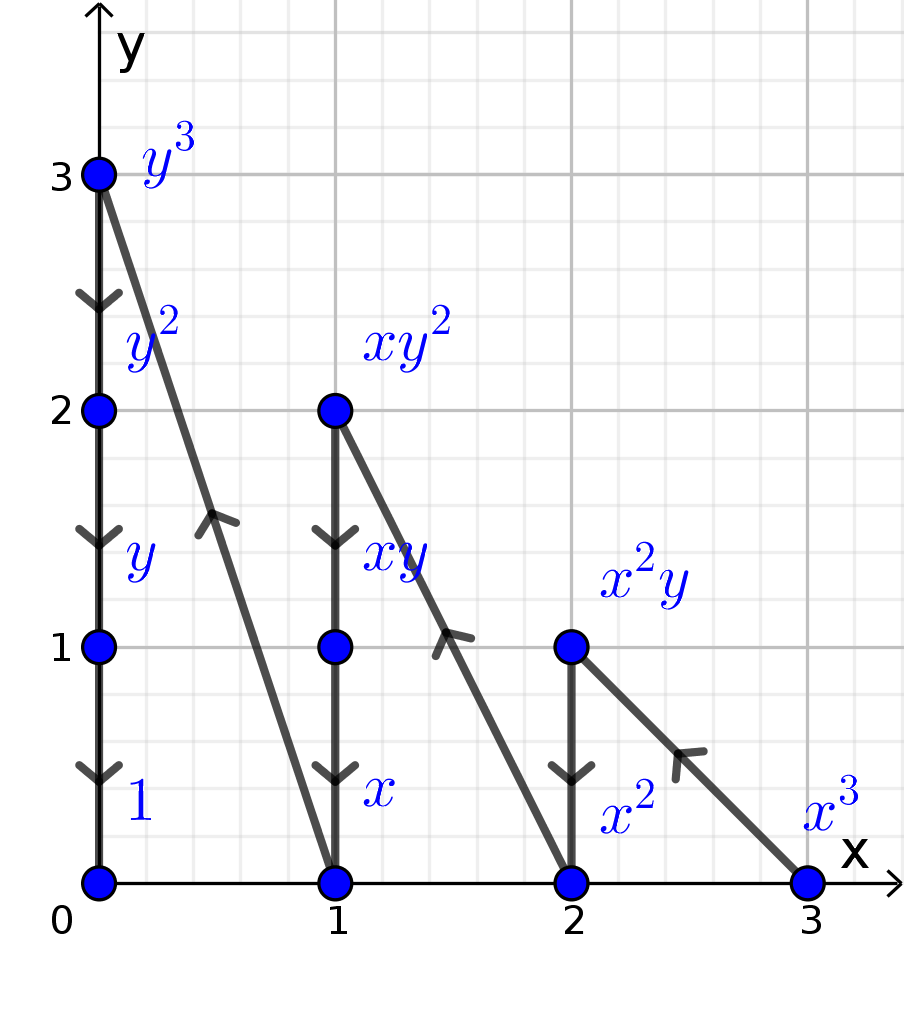}
\end{center}
\caption{Lexicographical term ordering for cubic two-variate polynomials}
\label{termordering}
\end{figure}

As homework two polynomial divisions were to be solved (see Fig.~\ref{hw-day4}).

\begin{figure}
\begin{center}
\includegraphics[width=0.7\textwidth]{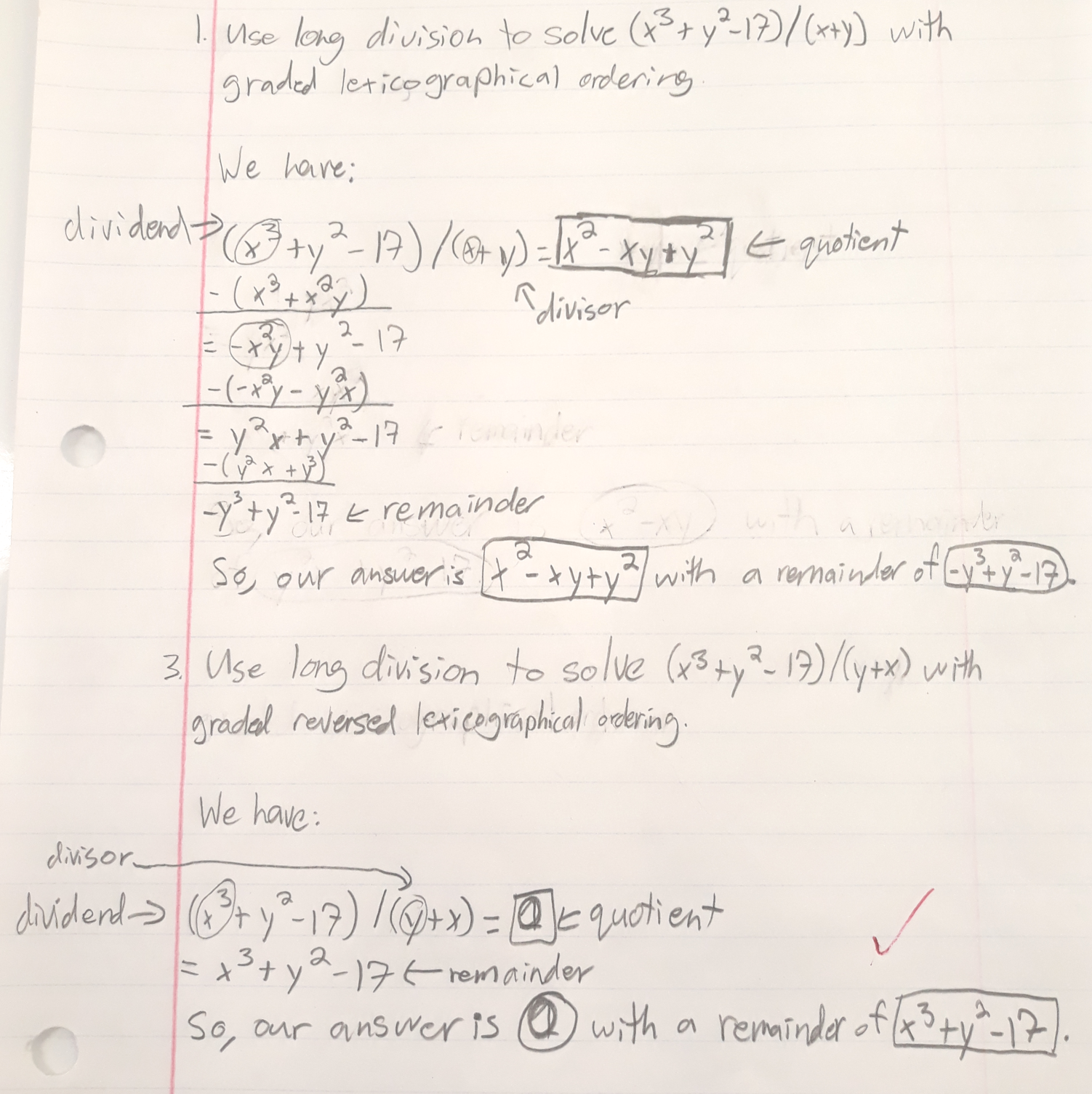}
\end{center}
\caption{A part of a student's completed homework for Day 4, in which the student learned that
the divisor is 0 in the second exercise and corrected the wrong answer
just in the moment of submission}
\label{hw-day4}
\end{figure}

\subsection*{Day 5}

Using Theorem \ref{quad-class} the quadratic polynomials were classified as
``atomic'' (irreducible) and ``non-atomic'' (reducible). It was highlighted
that $x^2-2y^2$ can be both if only rational coefficients or also irrationals
are allowed.

A student proof was given that the curve $x^3+y^2-17=0$ is atomic, based
on the fact that the sum of the degrees of factors is the degree of the curve,
and no linear factor can exist because if $x$ is big enough, the curve does
not include any points---and a vertical line is not included, either.

Several quartic polynomials were created by multiplying some linear and quadratic
components. The union of four lines $x=\pm1$ and $y=\pm1$ was given as homework
to get a degree 4 polynomial that describes all of them. A student question
was raised if the edge of the square $\{(x,y):|x|\leq1,\ |y|\leq1\}$
can be defined as a polynomial curve.

As a challenge homework it was asked how many intersection points exist if
we consider the curve $x^3+y^2-17=0$ and the line $x+y=0$. Also, a generalization
of this question was asked for an arbitrary cubic curve and an arbitrary line.
To prepare for a general answer a theorem was recalled:
\begin{theorem}\label{fta-m}
A polynomial in one variable of degree $n>0$ can have at most $n$ different roots.
\end{theorem}
This theorem was already known by the students. No proof was shown.

\subsection*{Day 6}

At the end of the first week some feedback was given by the students (see Fig.~\ref{adj1}). To spend more time
on the promised topic of LEGO linkages, from this point the lecturer decided to give at least
one linkage per day with building instructions to construct, to gamify the concepts
better.

\begin{figure}
\begin{center}
\includegraphics[width=0.9\textwidth]{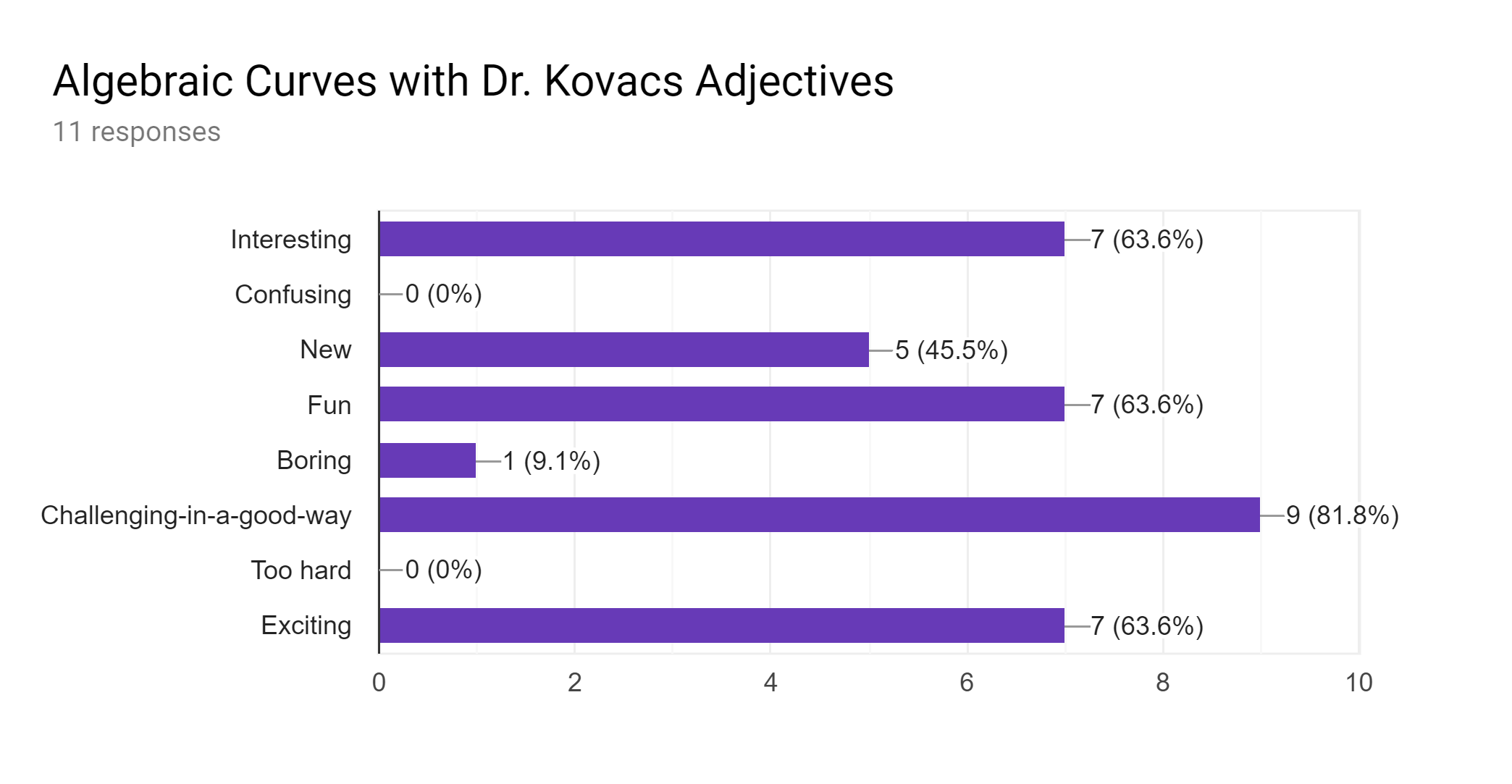}
\end{center}
\caption{Some responses of the students after the first week}
\label{adj1}
\end{figure}

An example of the ``Movement of a bug on a tooth paste box'' (\url{https://www.geogebra.org/m/qe699ja6})
was used to introduce the possibility of flipped bars and its consequence---namely, that the movement
can consist of more atomic movements and the corresponding polynomial can be factored.
(See \cite{gh-lego-linkages}, linkage
\texttt{Parallelogram-8-5} and \texttt{Parallelogram-4-8}; the latter was a homework to construct.)
Also, the problem of the movement of a cat that sits in the middle of a ladder that slides down the wall,
was given as a challenge homework---this question seemed to be difficult to answer: one student tried to
plot the positions of the cat but the graph was inaccurate, and the solution could not be determined.

The challenge question of Day 5 was discussed and a theorem was stated:

\begin{theorem}\label{Allison-sorrow}
If an algebraic curve contains a straight segment, then a whole line must also be included
(that is the lengthened segment).
\end{theorem}

\begin{proof}
Consider a polynomial $p$ that contains a straight segment that is a part of the line $\ell:ax+by+c=0$.
If $a\neq0$, $\ell$ can be described by the equation $x=(-c-by)/a$. Now substitute $x$ by this formula
in $p$, then a polynomial $q$ of $y$ will be obtained. Since $p=0$ holds for infinitely many $y$,
by using Theorem \ref{fta-m} we learn that $q$ must be 0 (otherwise there will be just a finite
number of zeroes). This means that $p=0$ holds for all $y$ and for all $x=(-c-by)/a$ which yields
the statement. The case $a=0$ can be verified analogously.
\end{proof}

To prepare for another theorem, a random quartic polynomial was divided by a linear divisor manually.
The process was slow and the result very complicated---but this step showed the computational
difficulty of long division in general. For checking the result GeoGebra's \texttt{Division}
command was used that outputs the quotient and the remainder in a list of two elements.

\subsection*{Day 7}

The solution of the sliding ladder problem \cite{ph,Hasek2018} was discussed. 
As homework, a question was raised about what would happen if a cat sat at one-third the way up a
the ladder. A student created a Python program (Fig.~\ref{py}) that solves the question graphically.
This question was later discussed in detail by using a GeoGebra applet
at \url{https://www.geogebra.org/m/dqgnvwn6}.

\begin{figure}
\begin{center}
\begin{lstlisting}
import turtle
import time
import math

screen = turtle.Screen()

cat = turtle.Turtle()
cat.shape("turtle")
cat.color("black")

def set_up():
  cat.speed(0)
  cat.goto(200,0)
  cat.goto(-200,0)
  cat.goto(0,0)
  cat.goto(0,-200)
  cat.goto(0,200)

set_up()
xcor = 0
add = 5

for ycor in range(200,-1,-1*add):
  cat.goto(0, ycor)
  cat.pendown()
  cat.goto(xcor/3, 2*ycor/3)
  cat.color("blue")
  cat.dot(4)
  cat.color("red")
  cat.goto(xcor, 0)
  cat.penup()
  xcor = math.sqrt(40000 - ycor*ycor)



print("Wait!")
time.sleep(2.5)
screen.reset()


set_up()
cat.goto(0,400/3)
cat.pensize(5)

accuracy = 1
xcor = 0

for cor in range(200*accuracy,-1,-1):
  ycor = cor/accuracy
  cat.goto(xcor/3, 2*ycor/3)
  cat.color("blue")
  xcor = math.sqrt(40000 - ycor*ycor)

cat.hideturtle()
print("Meow!")
\end{lstlisting}
\end{center}
\caption{A student's Python program that shows the locus of a cat while sitting on a sliding ladder}
\label{py}
\end{figure}

Today's LEGO construction was to build Chebyshev's linkage (see \cite{gh-lego-linkages}, linkage
\texttt{Chebyshev}). Some students conjectured that it produces a circular arc and a straight
line segment (if two bars of it are crossing). The homework was to verify this conjecture
or find a better one.

Some further computer experiments were done on 4-bar linkages, by changing the lengths of
the bars---we used the notion \textit{4-bar linkages} for ``planar linkages that involve two fixed
points and three bars'' (in some sense, the two fixed points lie on a fourth bar,
see \cite[p.~33]{bryantsangwin}).
For example, if a rhombus produced a union of three circles having radii 2, 2 and 4,
its algebraic form would be $ x^6 + y^6 + 3 x^2 y^4 + 3 x^4 y^2 - 12 x^5 - 12 x y^4 - 24 x^3 y^2 + 28 x^4 - 4 y^4 + 24 x^2 y^2 + 96 x^3 + 96 x y^2 - 272 x^2 - 144 y^2 - 192 x + 576 = 0$.
A challenge problem was to find its factors (but this problem was not solved by any students).

\subsection*{Day 8}

A simple special case of Hilbert's Nullstellensatz was proven:

\begin{theorem}\label{nss}
If an algebraic curve contains a straight line, then its polynomial has a linear factor.
\end{theorem}
\begin{proof}
Let $p$ be the polynomial of the curve. Suppose the line $\ell:ax+by+c=0$ is contained in the curve.
Divide $p$ by $ax+by+c$ with respect to the lexicographical term ordering. Now $p=(ax+by+c)\cdot q+r$
where $r$ is a polynomial of $y$. When plugging all points of $\ell$ in $p$ we always get $0$,
so---by using Theorem \ref{fta-m}---the remainder $r$ must be the constant $0$ polynomial.
That is, $ax+by+c$ is a factor of $p$.
\end{proof}

This is obviously the converse of the Zero Product Theorem (that is, Theorem \ref{th-factors}).

By setting up some equations the homework conjectures were checked by using a new technique for
manipulating a system of equations. First ``the cat on the ladder'' problem was solved by the equations
$a^2+b^2-1=0$, $x-a/3=0$ and $y-2b/3$, which were entered in
the online software tool \cite{gh-buchberger-singular} to obtain the following output:
\vskip12pt

\hrule
\noindent$\vdots\hfill\vdots$
\begin{quote}

The equation system is:
\begin{align*}a^{2}+b^{2}-1&= 0, &&(1)\\-\frac{1}{3}a+x&= 0, &&(2)\\-\frac{2}{3}b+y&= 0. &&(3)\\\end{align*}
Cancelling the leading term of (1) by (2) and continuing reduction:
\begin{align*}a^{2}+b^{2}-1&\underset{(2)}{\rightarrow}3ax+b^{2}-1\\ &\underset{(2)}{\rightarrow}\frac{1}{3}b^{2}+3x^{2}-\frac{1}{3}\\ &\underset{(3)}{\rightarrow}\frac{3}{2}by+9x^{2}-1\\ &\underset{(3)}{\rightarrow}6x^{2}+\frac{3}{2}y^{2}-\frac{2}{3}\sim \\ &\qquad x^{2}+\frac{1}{4}y^{2}-\frac{1}{9}.\\ \end{align*}
Equation added:
\begin{align*}x^{2}+\frac{1}{4}y^{2}-\frac{1}{9}&= 0. &&(4)\\\end{align*}
Cancelling the leading term of (1) by (3) and continuing reduction:
\begin{align*}a^{2}+b^{2}-1&\underset{(3)}{\rightarrow}\frac{3}{2}a^{2}y+b^{3}-b\\ &\underset{(1)}{\rightarrow}\frac{2}{3}b^{3}-b^{2}y-\frac{2}{3}b+y\\ &\underset{(3)}{\rightarrow}-b+\frac{3}{2}y\\ &\underset{(3)}{\rightarrow}0.\\ \end{align*}
Cancelling the leading term of (2) by (3) and continuing reduction:
\begin{align*}-\frac{1}{3}a+x&\underset{(3)}{\rightarrow}\frac{3}{2}ay-3bx\\ &\underset{(2)}{\rightarrow}-2bx+3xy\\ &\underset{(3)}{\rightarrow}0.\\ \end{align*}
Cancelling the leading term of (1) by (4) and continuing reduction:
\begin{align*}a^{2}+b^{2}-1&\underset{(4)}{\rightarrow}-\frac{1}{4}a^{2}y^{2}+\frac{1}{9}a^{2}+b^{2}x^{2}-x^{2}\\ &\underset{(1)}{\rightarrow}-\frac{4}{9}a^{2}-4b^{2}x^{2}-b^{2}y^{2}+4x^{2}+y^{2}\\ &\underset{(1)}{\rightarrow}9b^{2}x^{2}+\frac{9}{4}b^{2}y^{2}-b^{2}-9x^{2}-\frac{9}{4}y^{2}+1\\ &\underset{(3)}{\rightarrow}\frac{1}{4}b^{2}y^{2}-\frac{1}{9}b^{2}+\frac{3}{2}bx^{2}y-x^{2}-\frac{1}{4}y^{2}+\frac{1}{9}\\ &\underset{(3)}{\rightarrow}-\frac{4}{9}b^{2}+6bx^{2}y+\frac{3}{2}by^{3}-4x^{2}-y^{2}+\frac{4}{9}\\ &\underset{(3)}{\rightarrow}-\frac{27}{2}bx^{2}y-\frac{27}{8}by^{3}+\frac{3}{2}by+9x^{2}+\frac{9}{4}y^{2}-1\\ &\underset{(3)}{\rightarrow}\frac{1}{4}by^{3}-\frac{1}{9}by+\frac{3}{2}x^{2}y^{2}-\frac{2}{3}x^{2}-\frac{1}{6}y^{2}+\frac{2}{27}\\ &\underset{(3)}{\rightarrow}-\frac{4}{9}by+6x^{2}y^{2}-\frac{8}{3}x^{2}+\frac{3}{2}y^{4}-\frac{2}{3}y^{2}+\frac{8}{27}\\ &\underset{(3)}{\rightarrow}-\frac{27}{2}x^{2}y^{2}+6x^{2}-\frac{27}{8}y^{4}+3y^{2}-\frac{2}{3}\\ &\underset{(4)}{\rightarrow}-\frac{4}{9}x^{2}-\frac{1}{9}y^{2}+\frac{4}{81}\\ &\underset{(4)}{\rightarrow}0.\\ \end{align*}
Cancelling the leading term of (2) by (4) and continuing reduction:
\begin{align*}-\frac{1}{3}a+x&\underset{(4)}{\rightarrow}-\frac{1}{4}ay^{2}+\frac{1}{9}a-3x^{3}\\ &\underset{(2)}{\rightarrow}-\frac{4}{9}a+12x^{3}+3xy^{2}\\ &\underset{(2)}{\rightarrow}-27x^{3}-\frac{27}{4}xy^{2}+3x\\ &\underset{(4)}{\rightarrow}0.\\ \end{align*}
Cancelling the leading term of (3) by (4) and continuing reduction:
\begin{align*}-\frac{2}{3}b+y&\underset{(4)}{\rightarrow}-\frac{1}{4}by^{2}+\frac{1}{9}b-\frac{3}{2}x^{2}y\\ &\underset{(3)}{\rightarrow}-\frac{4}{9}b+6x^{2}y+\frac{3}{2}y^{3}\\ &\underset{(3)}{\rightarrow}-\frac{27}{2}x^{2}y-\frac{27}{8}y^{3}+\frac{3}{2}y\\ &\underset{(4)}{\rightarrow}0.\\ \end{align*}
\end{quote}
\noindent$\vdots\hfill\vdots$
\vskip6pt
\hrule

\vskip12pt

Here the atomic steps in each operation between two lines were simply explained, that is,
from $a^2+b^2-1$ one gets $3ax+b^2-1$ by computing $(a^2+b^2-1)+3a\cdot\left(-\frac13 a+x\right)$,
and so on. The leading term can always be cancelled by subtracting some appropriate multiples
of the two polynomials. Due to lack of time this algorithm was not further explained,
but the equation (4) was identified as a valid solution of the problem and it was found to
be a formula for an ellipse.

The Chebyshev linkage was similarly described by equations (see Fig.~\ref{bb})
and computed fully.

\begin{figure}
\begin{center}
\includegraphics[width=0.7\textwidth]{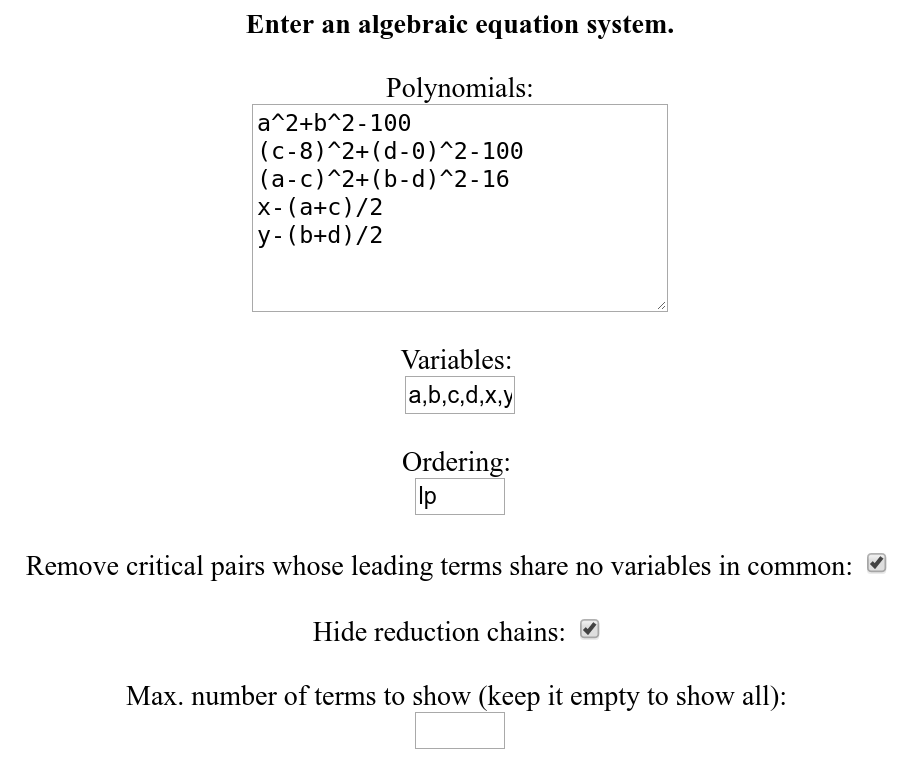}
\end{center}
\caption{Input form of the online software tool \texttt{buchberger-singular}}
\label{bb}
\end{figure}

The shocking amount of output lines and the magnitude of coefficients
confirmed the students that it is infeasible to do the same operations by hand. Luckily,
after 8 inserted equations one gets a polynomial in two variables:

\vskip12pt

\hrule
\noindent$\vdots\hfill\vdots$
\begin{quote}

The equation system is:
\begin{align*}a^{2}+b^{2}-100&= 0, &&(1)\\c^{2}-16c+d^{2}-36&= 0, &&(2)\\a^{2}-2ac+b^{2}-2bd+c^{2}+d^{2}-16&= 0, &&(3)\\-\frac{1}{2}a-\frac{1}{2}c+x&= 0, &&(4)\\-\frac{1}{2}b-\frac{1}{2}d+y&= 0. &&(5)\\\end{align*}
Cancelling the leading term of (1) by (3) and continuing reduction: $\ldots$

Equation added:
\begin{align*}cx-12c+dy-48&= 0. &&(6)\\\end{align*}
Cancelling the leading term of (1) by (4) and continuing reduction: $\ldots$

Equation added:
\begin{align*}c-\frac{1}{8}x^{2}-\frac{1}{8}y^{2}+8&= 0. &&(7)\\\end{align*}

Cancelling the leading term of (2) by (6) and continuing reduction: $\ldots$

Equation added:
\begin{align*}d^{2}x-12d^{2}-\frac{1}{8}dx^{2}y-\frac{1}{8}dy^{3}+24dy+6x^{2}-36x+6y^{2}-720&= 0. &&(8)\\\end{align*}

Cancelling the leading term of (2) by (7) and continuing reduction: $\ldots$

Equation added:
\begin{align*}d^{2}-\frac{1}{8}dxy-\frac{3}{2}dy+\frac{1}{64}x^{2}y^{2}-\frac{3}{4}x^{2}+6x+\frac{1}{64}y^{4}-\frac{7}{4}y^{2}+84&= 0. &&(9)\\\end{align*}

Cancelling the leading term of (6) by (7) and continuing reduction: $\ldots$

Equation added:
\begin{align*}dy+\frac{1}{8}x^{3}-\frac{3}{2}x^{2}+\frac{1}{8}xy^{2}-8x-\frac{3}{2}y^{2}+48&= 0. &&(10)\\\end{align*}

Cancelling the leading term of (8) by (10) and continuing reduction: $\ldots$

Equation added:
\begin{align*}dx^{4}-12dx^{3}+2dx^{2}y^{2}-64dx^{2}+384dx+dy^{4}-48dy^{2}\\
-\frac{3}{2}x^{2}y^{3}+24x^{2}y-288xy-\frac{3}{2}y^{5}+120y^{3}-2304y&= 0. &&(11)\\\end{align*}

Cancelling the leading term of (9) by (10) and continuing reduction: $\ldots$

Equation added:
\begin{align*}dx^{3}-12dx^{2}+2dxy^{2}-64dx+384d-\frac{1}{8}x^{2}y^{3}+6x^{2}y\\
-48xy-\frac{1}{8}y^{5}+14y^{3}-672y&= 0. &&(12)\\\end{align*}

Cancelling the leading term of (6) by (11) and continuing reduction: $\ldots$

Equation added:
\begin{align*}x^{7}y+3x^{5}y^{3}-560x^{5}y+24x^{4}y^{3}+2688x^{4}y+3x^{3}y^{5}\\
-1248x^{3}y^{3}+50176x^{3}y+48x^{2}y^{5}-172032x^{2}y+xy^{7}\\
-688xy^{5}+56064xy^{3}-1032192xy+24y^{7}\\
-2688y^{5}+18432y^{3}+3538944y&= 0. &&(13)\\\end{align*}
\end{quote}
\noindent$\vdots\hfill\vdots$
\vskip6pt
\hrule

\vskip12pt

The resulting equation is still of too high degree. Further operations can find a simpler polynomial:

\vskip12pt

\hrule
\noindent$\vdots\hfill\vdots$
\begin{quote}
Cancelling the leading term of (8) by (12) and continuing reduction: $\ldots$

Equation added:
\begin{align*}x^{6}-24x^{5}+3x^{4}y^{2}+16x^{4}-48x^{3}y^{2}+2304x^{3}+3x^{2}y^{4}\\
-96x^{2}y^{2}-5120x^{2}-24xy^{4}+2304xy^{2}\\
-49152x+y^{6}-112y^{4}+768y^{2}+147456&= 0. &&(17)\\\end{align*}
\end{quote}
\noindent$\vdots\hfill\vdots$
\vskip6pt
\hrule
\vskip12pt

This output can already be plotted in GeoGebra, and the graphical result confirms that no linear component exists:
by zooming in, the curve is slightly moving up and down, no straight vertical line can be observed.

A GeoGebra applet is available online at \url{https://www.geogebra.org/m/k4ywut2w}. Here GeoGebra's
\texttt{Eliminate} command was already explained and used, instead of the tool \cite{gh-buchberger-singular}.

The same idea was used to disprove the straight property of Watt's linkage. An online applet shows the details
at \url{https://www.geogebra.org/m/zvqb6nfc} (see also Fig.~\ref{Watt}).

\begin{figure}
\begin{center}
\includegraphics[width=0.9\textwidth]{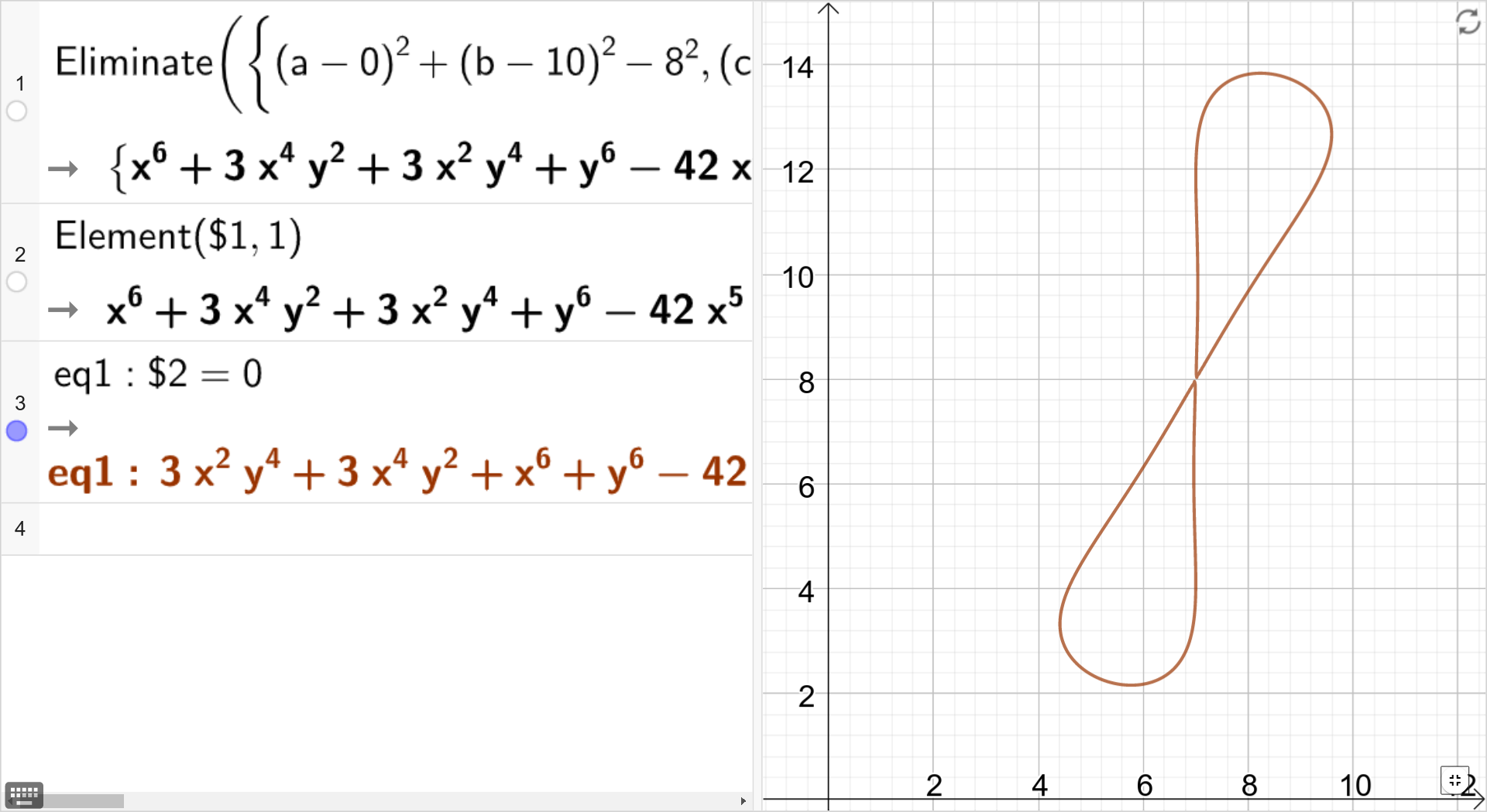}
\end{center}
\caption{Motion of Watt's linkage computed by GeoGebra's \texttt{Elimination} command}
\label{Watt}
\end{figure}

Here we silently assumed that GeoGebra's plotting capabilities were correct. However, it turned out that
on zooming in the curve of Watt's linkage, the crossing almost-straight parts are not drawn completely.
Therefore it was mentioned that for a full proof it should be shown that the obtained polynomials are atomic.

A linkage was to be constructed as homework: the \texttt{lambda} version of Chebyshev's linkage
(see \cite{gh-lego-linkages}).

\subsection*{Day 9}

The homework was to find the same result as Chebyshev's linkage. This was checked with
GeoGebra's \texttt{Elimination} command (\url{https://www.geogebra.org/m/hkpkqtfu}). Without proof the
following statement was highlighted:
\begin{theorem}\label{th-4bar}
A 4-bar linkage can produce either an atomic sextic movement, or a product of a quartic and a circular movement,
or a product of three circular movements.
\end{theorem}
Some former experiments from Days 6 and 7 were here recalled as an argument.

As a consequence we concluded that
\begin{theorem}\label{th-4bar2}
A 4-bar linkage cannot produce a straight line motion.
\end{theorem}

So the previous linkages are useless in our attempts. Therefore we need a different number of bars.
As a possible solution, Hart's inversor (\texttt{HartI}, see \cite{gh-lego-linkages}) was to be built as a LEGO activity, and to find an equation system
that describes it. A student found all but one (that is, 10) equations for the 12 variables. The homework
was to find the missing equation.

Also, another homework assignment was to build Hart's A-frame (\texttt{HartA2b}, see \cite{gh-lego-linkages}) and try to draw a straight line with it.

To get prepared for an ultimate linkage that produces a straight line motion, circle inversion was introduced.

\begin{definition}
Given a reference circle with center $O$ and radius $R$. The map of point $P$ is $P'$ if the points $O,P,P'$
are collinear and for the lengths $OP\cdot OP'=R^2$ holds.
\end{definition}

The ``most basic property'' of a circle inversion is that by doing it twice the identical map appears.
To learn some other basic properties of circle inversion, the idea \cite{aplimat-inv} was used. In particular,
a T-shirt that has horizontal stripes was used to illustrate the following statement:

\begin{theorem}\label{th-inv-lines}
A circle inversion maps a set of horizontal lines into a union of
\begin{itemize}
\item all those circles whose top/bottom point goes through then center of the reference circle,
\item a horizontal line that goes through the center of the reference circle.
\end{itemize}
\end{theorem}

A GeoGebra applet (available at \url{https://www.geogebra.org/m/kqta2rwa}) was used to have
a visual overview on this theorem. A homework was to prove the statement in that case when the line touches
the reference circle.

A second preparation for the final topic was to find an algebraic
equation to describe Agnesi's witch. Google's doodle
\url{https://www.google.com/doodles/maria-gaetana-agnesis-296th-birthday}
on Agnesi's 296th birthday was a visual introduction to this challenge.
By using the Intercept theorem a student managed to find an algebraic equation
(see \url{https://www.geogebra.org/m/zqbtzv5y}), but after elimination an
unwanted degenerate component also appeared. To avoid this, one finds Rabinowitsch's trick useful:

\begin{theorem}\label{rab}
The inequation $a\neq0$ is equivalent to the equation $a\cdot t-1=0$. 
\end{theorem}
In other words,
the inequation $a\neq0$ has a solution if and only if the equation $a\cdot t-1=0$ has a solution.
\begin{proof}
If $a\neq0$ then $t=1/a$. Otherwise if $a=0$ then $a\cdot t-1=0-1=-1\neq0$.
\end{proof}

By using Rabinowitsch's trick one can avoid the unwanted degenerate component as well
(see \url{https://www.geogebra.org/m/n8pxeaep} for a GeoGebra applet).

\subsection*{Day 10}

The missing equation in the algebraic description of Hart's inversor was inserted
and a polynomial $p$ was obtained by elimination. $p$ could be drawn in GeoGebra
and a linear component could be observed visually. Students discovered that
by using Theorem \ref{th-factors}, it can be checked if the linear component is
indeed included or not---the first suggestions were, however, to zoom out or zoom
in the graph to check the straight property from a distant or a closer perspective.
(See \url{https://www.geogebra.org/m/j5js3xx7}.)

The construction homework of Hart's A-frame was discussed quickly and a larger
variant of it was presented to the students to try out individually.
This activity was, however, technically difficult to perform and most students
simply skipped the individual check. A computer based check was also performed
with the students' help (\url{https://www.geogebra.org/m/smhm8mpc}). The long
division step was initiated by the students themselves for both linkages.

After these introductory activities the proof of a special case of Theorem \ref{th-inv-lines} was discussed
and given for homework: by using similar triangles and Thales' circle theorem the
statement is clear.

Then a converse of Theorem \ref{th-inv-lines} was stated:
\begin{theorem}\label{th-inv-circles}
A circle inversion maps all circles that go through the center of the reference
circle into a line.
\end{theorem}

\begin{proof}
Consider a circle $c$ that goes through the center $O$ of the reference circle.
There exists a line $\ell$ through $O$ that touches $c$. Rotate $\ell$ about $O$ to be horizontal
and rotate $c$ as well by the same angle. Now we have a line $\ell'$ and $c'$ after the rotation.
Clearly, $c'$ is a circle whose top/bottom point goes through the center of the reference circle.
By using Theorem \ref{th-inv-lines} there is a line $m$ that is mapped into $c'$. Therefore, because
of the ``most basic property'' of circle inversion, $c'$ is mapped into $m$. Since rotation
around the reference circle requires only a reverse rotation if a circle inversion was performed,
the map of $c$ must be a line, namely a reverse rotation of $m$ around $O$.
\end{proof}

At this point a linkage was built that was called \texttt{P-cell} (Fig.~\ref{pcell}). In fact, it referred to
Peaucellier's cell that can perform a circle inversion on two points that are not fixed
and not connected to three bars. The inversor property of Peaucellier's cell is well-known,
for example, see \url{https://en.wikipedia.org/wiki/Peaucellier%E2%80%93Lipkin_linkage#Inverse_points}
for a simple proof. During the course the same idea was explained.

\begin{figure}
\begin{center}
\includegraphics[width=0.9\textwidth]{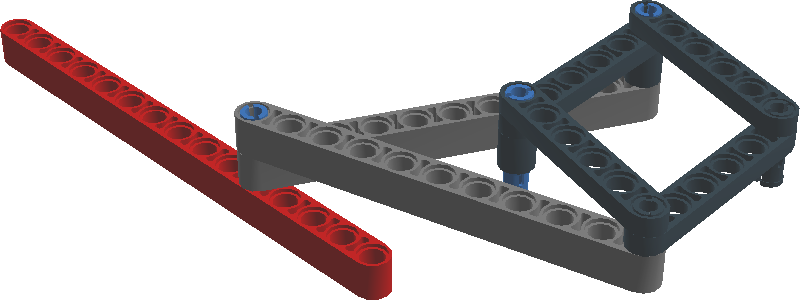}
\end{center}
\caption{Peaucellier's cell}
\label{pcell}
\end{figure}

A quiz question was raised in the classroom: How to extend the P-cell to obtain an exact
straight linear motion? A few students managed to solve this question by trial:
a bar with length 4 should be put on the right end of the long red bar and it should
connect one of the mutually mapped points to it (see Fig.~\ref{pcell-s}). Technically, it will constrain
the connected point to run on a circle. If the circle is going through the center
of the reference circle, by using Theorem \ref{th-inv-circles}, clearly the opposite point
of the rhombus will be running on a line.

\begin{figure}
\begin{center}
\includegraphics[width=0.6\textwidth]{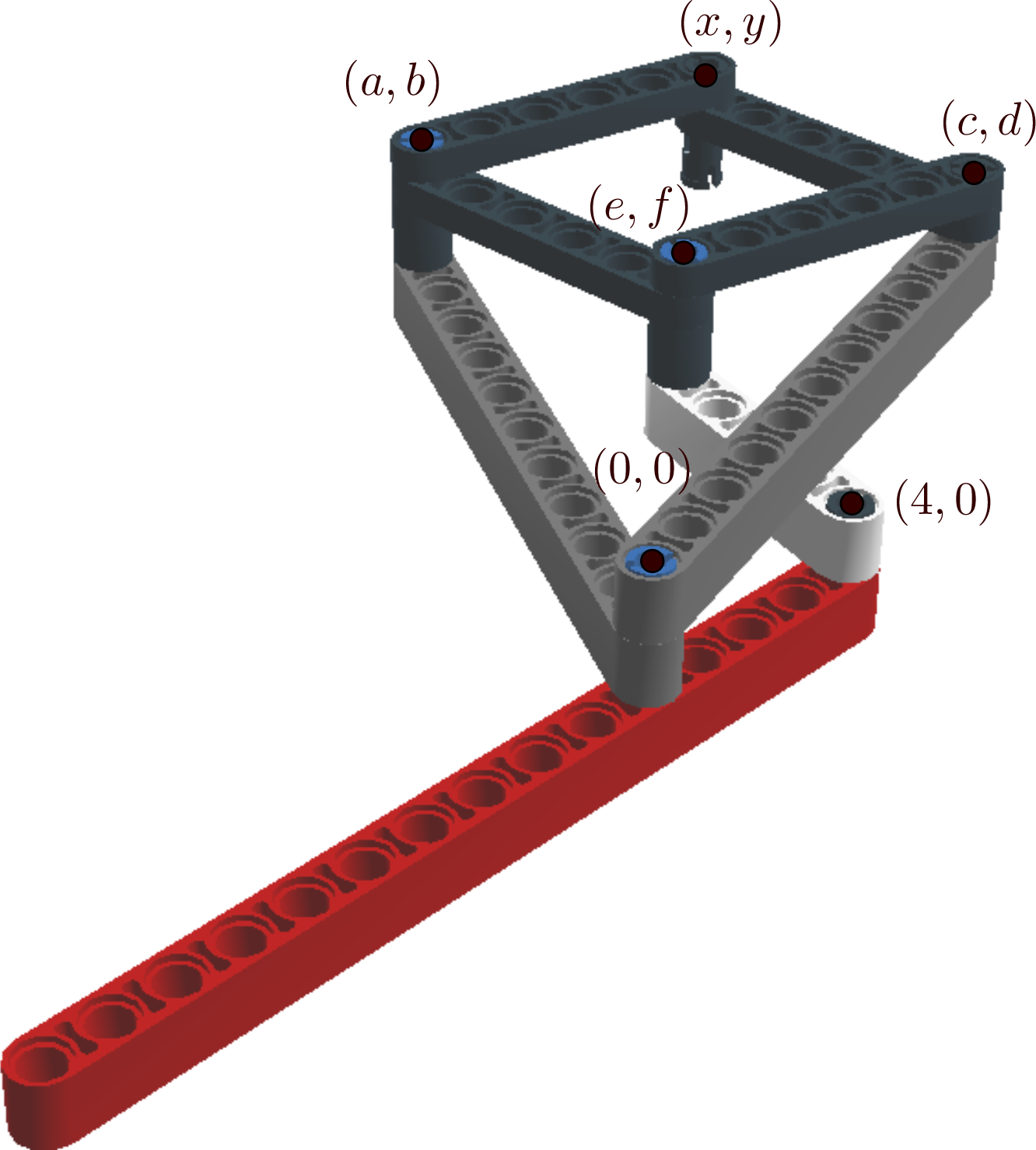}
\end{center}
\caption{The Peaucellier--Lipkin linkage}
\label{pcell-s}
\end{figure}

By setting up an equation system, however, the validity of this result cannot be directly
checked. That is, let us consider the equation system
\begin{align}
(e-4)^2+f^2-4^2=0,\\
(e-c)^2+(f-d)^2-5^2=0,\\
(c-x)^2+(d-y)^2-5^2=0,\\
(x-a)^2+(y-b)^2-5^2=0,\\
(a-e)^2+(b-f)^2-5^2=0,\\
a^2+b^2-10^2=0,\\
c^2+d^2-10^2=0.
\end{align}
By eliminating all variables but $x$ and $y$ no useful output will be obtained---the reason is that
the linkage moves ``too freely'' unless we consider some non-degeneracy conditions, namely, we
add $(e,f)\neq(x,y)$ and $(a,b)\neq(c,d)$. Thus, by adding the equations
\begin{align}
t((e-x)^2+(f-y)^2)-1=0,\\
u((a-c)^2+(b-d)^2)-1=0
\end{align}
we will obtain the equation $8x-75=0$
that clearly describes a straight line (see \url{https://www.geogebra.org/m/kehsjxws} for details).

\section{Conclusion}

Fig.~\ref{adj2} shows some student feedback after the second week. The main outcome of the course seemed
positive. However, there were some issues that could be improved in a future course.
\begin{figure}
\begin{center}
\includegraphics[width=0.9\textwidth]{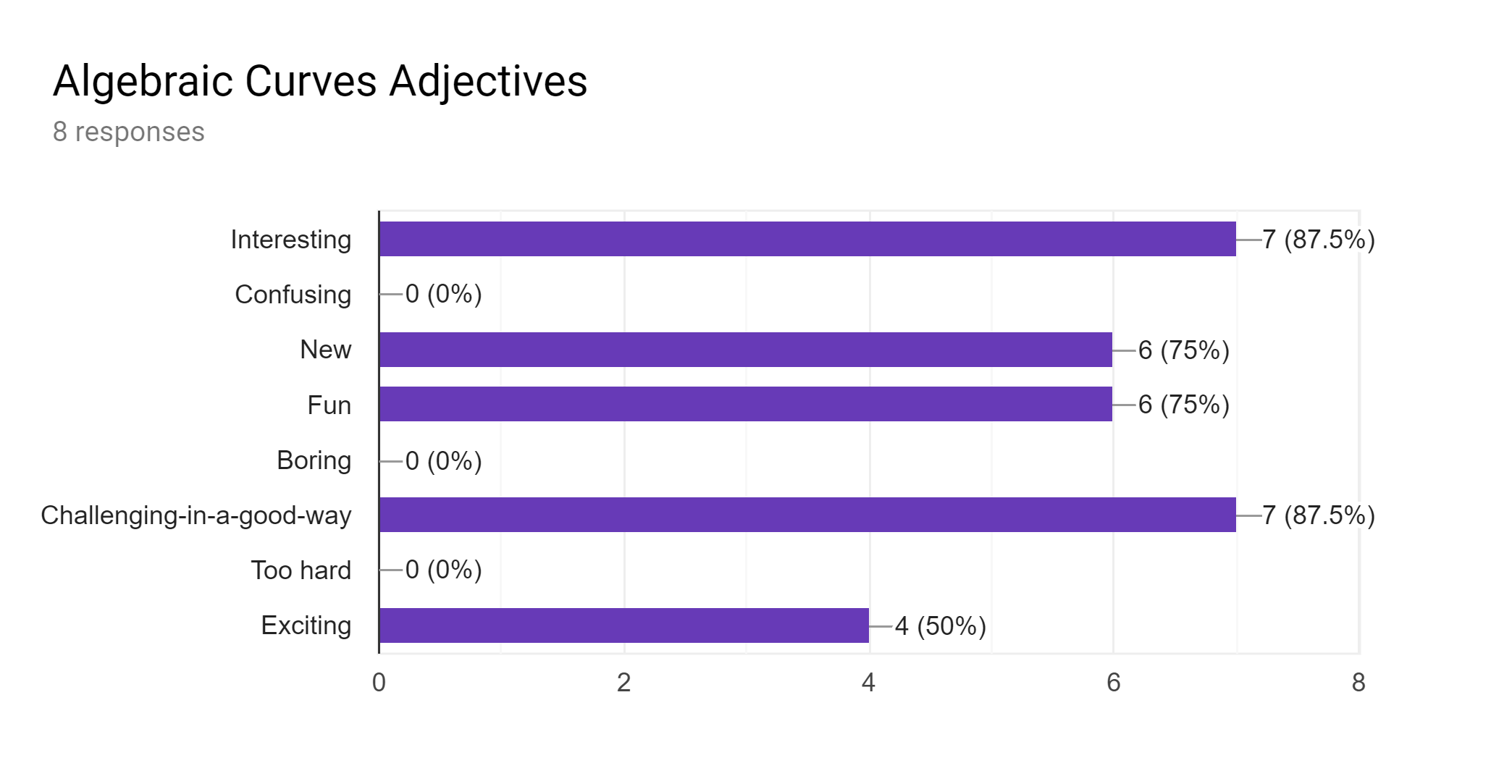}
\end{center}
\caption{Some responses of the students after the second week}
\label{adj2}
\end{figure}

First of all, drawing by utilizing the linkages was much more difficult for the students than it was expected.
The long red beam requires to stay in a fixed position while other parts of the linkage are moving. This
was difficult to perform for several students. As a result, the drawings were very inaccurate sometimes.

On the other hand, two-variate polynomial division was difficult for many students, even though it was
a prerequisite for entering the Gauss Level. In particular, the Algebra Assessment had included
a question on finding $(6x^5-x^4+3x^3+4x^2-x-1)/(3x^2+x-1)$.

In fact, Theorem \ref{nss} was not directly used during the course, so the whole explanation about
term orderings could have been omitted; however, they were still mentioned in the Buchberger algorithm
(because different term orderings can lead to completely different results). To see any purpose
of Theorem \ref{nss}, one should get an example of an irreducible higher degree polynomial that
seems to have a linear component. Both Chebyshev's and Watt's linkages could be such an example.
A simple way to show that, for example, Watt's linkage does not have a linear factor is to divide
the corresponding polynomial by $x-7$ and check that the remainder is not zero. Actually, by plugging $x=7$
in the polynomial we can directly use the proof of Theorem \ref{Allison-sorrow} and we can see the imperfection even
faster (see Fig.~\ref{Watt}).

Surprisingly enough, the students were not really interested in getting an automated result for
factoring a polynomial. GeoGebra's \texttt{Division} command was enough for all purposes. (Note that
the students did not have any computer access during the classroom hour except the teacher's,
and in general they did not have any laptops at camp.) Also, GeoGebra's \texttt{LocusEquation} command
was not used directly by the students---it appeared only in one applet when 4-bar linkages were
illustrated and the equation of their motion was automatically obtained.

At this point we need to admit that a formal proof of Theorem \ref{th-4bar} can be very difficult to find.
One approach can be to have different experiments on changing various parameters of the linkage
(including the positions of the bars) and classify the appearing curves after partitioning
the parameter space with the help of continuity. A similar idea is used in \cite{barker}.
Another approach can be found in \cite{muller}. In fact, there is a high variety of geometry of
the appearing coupler curves: the book \cite{4bar-atlas} illustrates more than
seven thousand different forms. 

Also, we need to emphasize that collecting experience on the algebraic form
of the coupler curves is very helpful in finding conjectures, even if formal proofs are
hard to achieve. For this purpose GeoGebra's \texttt{LocusEquation} command
could be an adequate tool. Learning and conjecturing that planar 4-bar linkages
always produce a sextic curve is a straightforward way by dragging the free points
in a dynamic geometry application.

Today's advanced mathematics knows a lot on how to directly compute the coupler curve by
using just the parameters of the linkages---the formulas are very complicated, though, see \cite{baiangeles}
for a recent study. But even if all coefficients of the sextic curve can be explicitly
computed, factorization of the achieved polynomial
cannot be easily performed. Therefore, nothing can be concluded on non-existence of linear factors.
In fact, further study of 4-bar linkages
will likely remain a field of active research for the coming decades, especially, since motion planning in robotics
has great importance and has direct connections with today's STEM education challenges.

\section{Acknowledgments}
The author is thankful to Tom\'as Recio for his suggestions on finalizing
the mathematical background of the course. Emily Castner kindly collected and
processed the students' feedback. Aryeh Stahl, counselor of the group of children,
gave useful feedback on the students' work. The author's ideas were supported
by several colleagues, including Csaba Biro, Beatrix Hauer, Andreas Kiener,
Chris Sangwin, Susanne Thrainer and M.~Pilar V\'elez. Some useful feedback on the
first versions of this paper was
given by Allison Hung and Noah Mok. Special thanks to Jonathan Yu for several
suggestions that improved the text substantially.

Tom Edgar suggested considering the paper \cite{groebner-nim} that attempts
introducing the Buchberger algorithm by using a solitaire game. In a future version
of the course this idea could be indeed used to learn more about the basics of the
theory of Gr\"obner bases.

The project was partially supported by a grant MTM2017-88796-P from the
Spanish MINECO (Ministerio de Economia y Competitividad) and the ERDF
(European Regional Development Fund).

\bibliography{kovzol,external}

\begin{thebibliography}{10}

\bibitem{gh-lego-linkages}
Kov\'acs, Z.:
\newblock lego-linkages.
\newblock A GitHub project (2019)
  \url{https://github.com/kovzol/lego-linkages}.

\bibitem{nullstellensatz}
Hilbert, D.:
\newblock {\"Uber} die vollen {Invariantensysteme}.
\newblock Mathematische {Annalen} \textbf{42} (1893)  313--337

\bibitem{Singular}
Decker, W., Greuel, G.M., Pfister, G., Sch{\"o}nemann, H.:
\newblock {{\sc Singular} {4-1-2} --- {A} computer algebra system for
  polynomial computations}.
\newblock (2019) \url{http://www.singular.uni-kl.de}.

\bibitem{gb-en}
Buchberger, B.:
\newblock {B}runo {B}uchberger's {PhD} thesis 1965: An algorithm for finding
  the basis elements of the residue class ring of a zero dimensional polynomial
  ideal.
\newblock Journal of Symbolic Computation \textbf{41} (2006)  475--511

\bibitem{gg5}
Hohenwarter, M., Borcherds, M., Ancsin, G., Bencze, B., Blossier, M.,
  Delobelle, A., Denizet, C., \'Eli\'as, J., Fekete, A., G\'al, L.,
  Kone\v{c}n\'y, Z., Kov\'acs, Z., Lizelfelner, S., Parisse, B., Sturr, G.:
\newblock {G}eo{G}ebra 5 (2014) \url{http://www.geogebra.org}.

\bibitem{Rabinowitsch1929}
Rabinowitsch, J.:
\newblock Zum {H}ilbertschen {N}ullstellensatz.
\newblock Math. Ann. \textbf{102} (1929)  520

\bibitem{Kempe1877}
Kempe, A.B.:
\newblock How to Draw a Straight Line; a Lecture on Linkages.
\newblock MacMillan and Co., London (1877)

\bibitem{bryantsangwin}
Bryant, J., Sangwin, C.:
\newblock How round is your circle? Where engineering and mathematics meet.
\newblock Princeton, New Jersey: Princeton University Press (2008)

\bibitem{lego}
Kov{\'{a}}cs, Z., Kov{\'{a}}cs, B.:
\newblock A compilation of {LEGO} {Technic} parts to support learning
  experiments on linkages.
\newblock CoRR \textbf{abs/1712.00440v2} (2017)

\bibitem{ucmlego}
Kov\'acs, Z.:
\newblock Ense\~nando geometr\'{\i}a y \'algebra a trav\'es de mecanismos y
  {LEGO}.
\newblock Presentation at STEM Seminar, UCM, Madrid, Spain (2019)

\bibitem{jkulego}
Kov\'acs, Z.:
\newblock Teaching algebra and geometry with {LEGO} and linkages.
\newblock Presentation at STEAM Seminar, JKU, Linz, Austria (2019)

\bibitem{vmt2}
Kov\'acs, Z.:
\newblock Motion with {LEGO}s and dynamic geometry.
\newblock Virginia Mathematics Teacher Journal \textbf{44} (2018)  43--48

\bibitem{JSC-linkages}
Kov\'acs, Z., Recio, T., V\'elez, M.P.:
\newblock Reasoning about linkages with dynamic geometry.
\newblock Journal of Symbolic Computation (2019)

\bibitem{aplimat-inv}
Kov\'acs, Z.:
\newblock Teaching inversion interactively with webcams via {{CindyJS}}.
\newblock In: 18th conference on applied mathematics. Aplimat 2019 Proceedings,
  Slovak University of Technology (2019)  691--699

\bibitem{lego-steam}
Kov{\'a}cs, Z., Kov\'acs, B.:
\newblock A compilation of {LEGO} {Technic} parts to support learning
  experiments on linkages.
\newblock Presentation at Linz STEAM Education Conference, Linz, Austria (2018)

\bibitem{conicsections-irina}
Boyadzhiev, I.:
\newblock Conic sections.
\newblock GeoGebra Materials (2015) \url{https://www.geogebra.org/m/T8TV2JqG}.

\bibitem{ph}
Hauer, B., Kov{\'{a}}cs, Z., Recio, T., V\'elez, M.P.:
\newblock Automated reasoning in elementary geometry: towards inquiry learning.
\newblock Pedagogical Horizons \textbf{2} (2018)  1--13

\bibitem{Hasek2018}
Ha\v{s}ek, R.:
\newblock Dynamic geometry software supplemented with a computer algebra system
  as a proving tool.
\newblock Mathematics in Computer Science (2018)
  \url{https://doi.org/10.1007/s11786-018-0369-x}.

\bibitem{gh-buchberger-singular}
Kov\'acs, Z.:
\newblock buchberger-singular.
\newblock A GitHub project (2019)
  \url{https://github.com/kovzol/buchberger-singular}.

\bibitem{barker}
Barker, C.R.:
\newblock A complete classification of planar four-bar linkages.
\newblock Mechanism and Machine Theory \textbf{20} (1985)  535--554

\bibitem{muller}
Muller, M.:
\newblock A novel classification of planar four-bar linkages and its
  application to the mechanical analysis of animal systems.
\newblock Philosophical Transactions: Biological Sciences \textbf{351} (1996)
  689--720

\bibitem{4bar-atlas}
Hrones, J.A., Nelson, G.L.:
\newblock Analysis of the Four-bar Linkage: Its Application to the Synthesis of
  Mechanisms.
\newblock Technology Press of MIT, Cambridge (1951)

\bibitem{baiangeles}
Bai, S., Angeles, J.:
\newblock Coupler-curve synthesis of four-bar linkages via a novel formulation.
\newblock Mechanism and Machine Theory \textbf{94} (2015)  177--187

\bibitem{groebner-nim}
Dozier, H., Perry, J.:
\newblock Androids armed with poisoned chocolate squares: Ideal nim and its
  relatives.
\newblock Mathematics Magazine \textbf{89} (2016)  235--250

\end{thebibliography}

\end{document}